\documentclass[12pt]{article}
\usepackage{amsmath}
\usepackage{amssymb}
\usepackage{amsthm}
\usepackage{pstricks}

\title{Automorphism Groups of Schur Rings}
\author{Brent Kerby}

\DeclareMathOperator{\Aut}{Aut} 
 \DeclareMathOperator{\Char}{Char}
 
\DeclareMathOperator{\Sym}{Sym}

\newtheorem{thm}{Theorem}[section]
\newtheorem{conj}[thm]{Conjecture}
\newtheorem{lemma}[thm]{Lemma}
\newtheorem{cor}[thm]{Corollary}
\newtheorem*{mainthm}{Theorem \ref{mainthm}}
\newtheorem*{mainthm2}{Theorem \ref{autcyc}}
\newtheorem*{mainthm3}{Theorem \ref{thm_conv}}
\theoremstyle{definition}
\newtheorem{defn}[thm]{Definition}
\newtheorem{example}[thm]{Example}
\newtheorem{problem}[thm]{Problem}

\theoremstyle{remark}
\newtheorem*{remark}{Remark}

\newcommand{\h}[1]{\mathbf{#1}}
\newcommand{\q}[1]{\overline{#1}}
\newcommand{\m}[1]{\mathcal{#1}}

\newcommand{\Z}{\mathbb Z}

\begin{document}
\maketitle

\begin{center}
Brigham Young University,
Department of Mathematics,
Provo, UT 84604 
\end{center}

\begin{abstract}
In 1993, Muzychuk \cite{muzychuk} showed that the rational Schur rings over a cyclic
group $Z_n$ are in one-to-one correspondence with sublattices of the
divisor lattice of $n$, or equivalently, with sublattices of the
lattice of subgroups of $Z_n$. This can easily be extended to show
that for any finite group $G$, sublattices of the lattice of
characteristic subgroups of $G$ give rise to rational Schur rings over
$G$ in a natural way. Our main result is that any finite group may
be represented as the (algebraic) automorphism group of such a rational Schur ring
over an abelian $p$-group, for any odd prime $p$. In contrast, we show that over a cyclic group the automorphism group of any Schur ring is abelian. We also prove a converse
to the well-known result of Muzychuk \cite{muzychuk2} that two Schur rings over a cyclic
group are isomorphic if and only if they coincide; namely, we show
that over a group which is not cyclic, there always exist distinct
isomorphic Schur rings.
\end{abstract}

Schur rings (or S-rings), first developed by I. Schur in \cite{schur} in 1933 as a tool for investigating permutation groups, have subsequently found several applications in graph theory and algebraic combinatorics. For a survey of the recent developments in S-rings and their applications, see \cite{muzychuk_ponomarenko}. All of the definitions and elementary results that we need will be reviewed in \S\ref{sec_def}.

A great deal of effort has been focused on obtaining a classification of S-rings over cyclic groups, which was achieved in \cite{leung_man,leung_man2}. After the successful classification of S-rings over cyclic groups, a logical next step has been to seek classifications for S-rings over broader families of groups. This, however, appears to be a very difficult problem even in the case, for instance, of abelian $p$-groups. One approach is to restrict attention to special types of S-rings, in particular rational S-rings, in the hope that understanding these may shed light on the general problem. In \S\ref{sec_rat} we describe a general construction which, for each sublattice of the lattice of characteristic subgroups of a group $G$, produces a corresponding rational S-ring over $G$. This construction generalizes the one in \cite{muzychuk} used to classify rational S-rings over cyclic groups. In \S\ref{sec_auto} we review the structure of the lattice of characteristic subgroups over abelian $p$-groups. This enables us to explicitly describe the construction of many rational S-rings over such groups, which we will need for our main result in \S\ref{sec_main}.

We consider the following question: Which groups can occur as the automorphism group of an S-ring? It is clear that the automorphism group of a full group ring $\mathbb{C}G$ (as an S-ring) is isomorphic to the automorphism group of $G$. Thus, any group which may be represented as the automorphism group of a group may also be represented as the automorphism group of an S-ring in a rather trivial way. It is known, however, that there are finite groups (for instance, $Z_3$) which are not isomorphic to the automorphism group of any group (see, e.g., \cite{devries}, or \cite[Theorem 3.6]{kerby_masters}). Nevertheless, every finite group occurs as the automorphism group of some S-ring, as our main result shows:

\begin{mainthm}[Main Theorem]
Let $p$ be an odd prime. Every finite group can be represented as the automorphism group of a rational S-ring over an abelian $p$-group.
\end{mainthm}

At the core of the proof is an application of a theorem of Birkhoff that any finite group can be represented as the automorphism group of a finite distributive lattice. In \S\ref{sec_cyc} we use the Leung-Man classification of S-rings over cyclic groups to prove the following theorem, which shows that it is impossible to strengthen our main result by restricting to S-rings over cyclic groups:

\begin{mainthm2}
If $S$ is an S-ring over a cyclic group $Z_n$, then $\Aut(S)$ is abelian.
\end{mainthm2}

A remarkable theorem of Muzychuk states that over a cyclic group, two S-rings are isomorphic if and only if they coincide \cite{muzychuk2}. This fact plays a key role in several applications, including our proof of Theorem \ref{autcyc}. In \S\ref{sec_conv}, we prove a converse to this result, as follows:
\begin{mainthm3}
Let $G$ be a finite group which is not cyclic. Then there exist distinct Cayley-isomorphic S-rings $S_1$ and $S_2$ over $G$.
\end{mainthm3}

Finally, in \S\ref{sec_ex} we mention some interesting examples of rational S-rings over abelian $p$-groups which cannot be constructed using the method of \S\ref{sec_rat}. We do not yet entirely understand how to generalize these examples, but doing so may be a starting point for approaching the general problem of classifying rational S-rings over abelian $p$-groups, a problem which we believe shows some promise of being tractable.

\ \\
Note: Most of the results in this paper were included as part of the author's master's thesis  \cite{kerby_masters}, under the supervision of Stephen P. Humphries.

\section{Basic definitions}\label{sec_def}

Let $G$ be a finite group and $F$ a field. If $C$
is a subset of $G$, then we define $\q C$ to be the element $\sum_{g
\in C}g$ in the group algebra $FG$, and we call $\q C$ a \emph{simple quantity} of $FG$. Given a subset $C \subseteq G$ and an integer $m$, we define
$C^{(m)}=\{g^m : g \in C\}$. For any $x \in FG$, where $x =
\sum_{g\in G} a_gg$, we define $x^{(m)} = \sum_{g\in G}a_gg^m$.
Given a subset $X \subseteq FG$, we denote the subspace spanned by $X$ by
$FX$. If $\m L$ is a collection of subsets of $G$, we let $\q{\m L}$ denote the set $\{\q C : C \in \m L\}$.

A subalgebra $S$ of the group algebra $FG$ is called a \emph{Schur ring} (or \emph{S-ring}) over $G$ if there are disjoint nonempty subsets $T_1,\dots,T_n$ of $G$ such that $\q T_1,\dots, \q T_n$ form a basis for $S$, with the following properties,
\begin{enumerate}
\item[(i)] For every $i$ there is some $j$ such that $T_i^{(-1)}=T_j$.
\item[(ii)] $T_1=\{1\}$, and $G=T_1\cup T_2\cup \cdots \cup T_n$.
\end{enumerate}

The sets $T_1,\dots,T_n$ are called \emph{basic sets} of $S$ and are said to
form a \emph{Schur partition} of $G$. The corresponding $\q
T_1,\dots,\q T_n$ are called the \emph{basic quantities} of $S$. If $S$
satisfies condition (i), but perhaps not (ii), then $S$ is
called a \emph{pseudo S-ring} (or \emph{PS-ring}). The two sets $\{1\}$, $G\setminus\{1\}$ always form a Schur partition; the corresponding S-ring is called the \emph{trivial} S-ring over $G$.

Given $x, y \in FG$, where $x=\sum_{g \in G}a_gg,y=\sum_{g\in G}b_gg$, their \emph{Hadamard product} is defined by
$$x \circ y = \sum_{g\in G}a_gb_gg.$$
There is a well-known purely algebraic description of S-rings and PS-rings in terms of closure under the Hadamard product, avoiding reference to the combinatorial notion of basic sets (for proofs, see \cite[Proposition 3.1]{muzychuk} and \cite[Lemma 1.3]{muzychuk2}, or \cite[Theorem 1.7, Corollary 1.8]{kerby_masters}):

\begin{thm}\label{folklore}
Let $A$ be a subalgebra of $FG$. Then $A$ is a $PS$-ring if and
only if $A$ is closed under $\circ$ and $^{(-1)}$.
\end{thm}

\begin{cor}\label{folklore_cor}
Let $A$ be a subalgebra of $FG$. Then $A$ is an $S$-ring if and only if $A$ is closed under $\circ$ and $^{(-1)}$ and contains $1$ and $\q G$.
\end{cor}

\begin{defn}
Let $S_1$ and $S_2$ be S-rings over groups $G_1$ and $G_2$
respectively. An algebra isomorphism $\phi : S_1 \to S_2$ is called
an \emph{isomorphism} of S-rings if $\phi$ maps basic quantities to
basic quantities, i.e., if for every basic set $C$ of $S_1$ there is
some basic set $D$ of $S_2$ such that $\phi(\q C)=\q D$.
\end{defn}

The following well-known result gives a purely algebraic characterization of S-ring isomorphisms (an elementary proof may be found in \cite[Theorem 3.4]{kerby_masters}):

\begin{thm}\label{isohad}
Let $S_1$ and $S_2$ be S-rings over groups $G_1$ and $G_2$
respectively. Then an $F$-algebra isomorphism $\phi: S_1 \to S_2$ is
an S-ring isomorphism if and only if $\phi$ respects the Hadamard
product, i.e., if and only if $\phi(x \circ y) = \phi(x) \circ
\phi(y)$ for all $x,y \in S_1$.
\end{thm}

Elsewhere in the literature, for clarity, an isomorphism of S-rings is sometimes called an \emph{algebraic isomorphism}, in contrast to the notion of a \emph{combinatorial isomorphism} of S-rings, which, properly speaking, is an isomorphism of the association schemes corresponding to the S-rings.

If $\q C$ is any simple quantity contained in an S-ring $S$, then we say that $C$ is an \emph{$S$-set}. The following is also well-known (For a proof, see \cite[Theorem 3.7, Corollary 6.8]{kerby_masters}):

\begin{thm}\label{isosimp}
Let $\phi: S_1 \to S_2$ be an isomorphism of S-rings over groups $G_1$ and
$G_2$ respectively.
Let $C \subseteq G_1$ be an $S_1$-set. Then $\phi(\q C)=\q D$ for some $S_2$-set $D \subseteq G_2$, and we write $\phi(C)=D$. Moreover,
\begin{enumerate}
\item[(i)] $|C|=|D|$, and
\item[(ii)] If $C$ is a subgroup then $D$ is also a subgroup.
\end{enumerate}
\end{thm}

An isomorphism from an S-ring onto itself is called an \emph{automorphism}. The set of automorphisms of an S-ring $S$ is denoted $\Aut(S)$. Every group automorphism in
$\Aut(G)$ naturally induces an S-ring isomorphism of $S$ (onto a possibly distinct S-ring);  such an isomorphism is called a \emph{Cayley isomorphism}, and the two S-rings are called \emph{Cayley-isomorphic}.

It is possible to consider S-rings where the coefficient field $F$ is replaced by an arbitrary ring $R$ (with unity). The collection of Schur partitions is affected only by the characteristic of $R$; so whether $R$ is $\mathbb{Z}$, $\mathbb{Q}$, $\mathbb{C}$, or any other ring of characteristic 0 makes no difference, as far as the collection of Schur partitions is concerned. Nevertheless, Theorem \ref{folklore} and Corollary \ref{folklore_cor} both become false over any commutative ring $R$ which is not a field; for this reason, we will always take our coefficient ring to be a field $F$. Over rings of nonzero characteristic, the collection of Schur partitions includes all those associated with zero characteristic in addition possibly to others. (For proof of these statements, see \cite[Example 1.10, Theorem 1.12]{kerby_masters}.) It is possible for the isomorphism class of the automorphism group $\Aut(S)$ to change depending on the characteristic of the coefficient field $F$ (see Example \ref{ex_nzchar} below). Throughout the literature on S-rings, most authors assume a coefficient ring of characteristic zero. We will also need this assumption in \S\ref{sec_cyc} when we show that the automorphism group of an S-ring over a cyclic group is abelian; indeed, this theorem becomes false over any ring of nonzero characteristic, as Example \ref{ex_nzchar} will show.

\section{Central and rational S-rings}\label{sec_rat}

\begin{defn}
A PS-ring $S$ over a group $G$ is \emph{central} if $S \subseteq
Z(FG)$, i.e., $S$ is contained in the center of the group algebra.
\end{defn}
\begin{remark}
This is equivalent to requiring that every basic set $T_i$ be a union of conjugacy classes of $G$.
\end{remark}

Of course, over an abelian group every S-ring is central. We are
primarily interested in a special class of central S-rings known
as rational S-rings:

\begin{defn}
A PS-ring $S$ over a group $G$ is \emph{rational} if for every $x\in
S$ and $\phi \in\Aut(G)$, we have $\phi(x)=x$.
\end{defn}
\begin{remark}
This is equivalent to requiring that every basic set $T_i$ be a union of automorphism
classes of $G$, where the \emph{automorphism classes} of $G$ are the orbits of $\Aut(G)$ acting on $G$ in the natural way. 
\end{remark}

The following provides a method for constructing many interesting central and rational S-rings:

\begin{thm}\label{latsring}
Let $G$ be any finite group, and let $\m{L}$ be any sublattice of
the lattice of normal subgroups of $G$. Then the vector space $F\q{\m
L}$ is a central PS-ring over $G$ with the
following properties, for all $H, K \in \m L$:
\begin{enumerate}
\item[(i)]$\q H^{(-1)} = \q H$
\item[(ii)] $\q H \circ \q K = \q{H\cap K}$
\item[(iii)] $\q H\ \q K = |H\cap K|\q{HK}$
\item[(iv)] $F\q{\m L}$ is an S-ring if and only if $1, G \in \m L$.
\item[(v)] $F\q{\m L}$ is rational if and only if $\m L$ consists entirely of characteristic subgroups.
\end{enumerate}
\end{thm}
\begin{proof}
(i) is clear since $H$, as a subgroup, is closed under inverses.
(ii) is immediate from the definition of the Hadamard product. (iii)
is clear since, by elementary group theory, every element of $HK$
can be written in $|H\cap K|$ ways as a product of an element in $H$
with an element in $K$. Now, since the subgroups $H$ and $K$ are
normal, $HK$ is also a subgroup of $G$; since $\m L$ is a lattice,
we have $H\cap K, HK \in\m L$. Thus, (i)--(iii) show that $F\q{\m L}$
is closed under $^{(-1)}$, multiplication, and the Hadamard product,
so by Theorem \ref{folklore}, $F\q{\m L}$ is a PS-ring. Since each
subgroup $H$ is normal, we have $g\q Hg^{-1}=\q{gHg^{-1}}=\q H$ for
all $g\in G$ and $H \in\m L$, so that $\q H \in Z(FG)$. It follows
that $F\q{\m L}$ is a central PS-ring.

Now, if $1, G \in \m L$ then $1, \q G \in F\q{\m L}$, so by Corollary
\ref{folklore_cor}, $F\q{\m L}$ is an S-ring. Suppose conversely
that $F\q{\m L}$ is an S-ring. Let $L=\bigcap_{H\in\m L} H$, so that
$L \in \m L$. If $|L|>1$, then since $L \subseteq H$ for every spanning
element $\q H$ of $F\q{\m L}$, it follows that every element of $F\q{\m L}$ with a non-zero coefficient of $1\in G$ also has other non-zero
coefficients (namely, the other elements of $L$ have non-zero
coefficients), so that $1 \notin F\q{\m L}$, contrary to the assumption that $F\q{\m L}$ is an S-ring. So we must have $L=1$, hence $1 \in \m L$, as desired. Now define
$M=\prod_{H\in\m L} H$. If $M\neq G$, then since $H \subseteq M$ for
every spanning element $\q H$ of $F\q{\m L}$, it follows that the nonzero
coefficients of every element $x \in F\q{\m L}$ are contained in $M$,
so that $\q G \notin F\q{\m L}$, again contradicting that $F\q{\m L}$ is an S-ring. So $G \in \m L$ as desired.

Finally, $F\q{\m L}$ is rational if and only if $\phi(\q H)=\q H$ for
every $\phi\in\Aut(G)$ and each spanning element $\q H$ of $F\q{\m L}$;
this holds if and only if $\phi(H)=H$ for each $H\in\m L$, i.e. if
and only if each $H\in\m L$ is characteristic.
\end{proof}

If $G$ is a cyclic group, then every subgroup of $G$ is
characteristic; consequently, the construction of Theorem
\ref{latsring} produces only rational S-rings. In this context, we
may state the main theorem of \cite{muzychuk} as follows:

\begin{thm}\emph{(Muzychuk)}\label{ratcyc}Every rational S-ring over a finite cyclic group may be constructed as in Theorem \ref{latsring}.
\end{thm}

There are rational S-rings over abelian $p$-groups
which cannot be constructed as in Theorem \ref{latsring} (see
Example \ref{nonlatsring}). However, there are other types of groups
for which Theorem \ref{latsring} produces the complete set of
rational S-rings. We mention an example, whose proof can be found in \cite[Theorem 2.6]{kerby_masters}:

\begin{thm}\label{ratdih}Every rational S-ring over a finite dihedral group may be constructed as in Theorem \ref{latsring}.
\end{thm}

\section{Characteristic subgroups of abelian $p$-groups}\label{sec_auto}

In this section, we review the description of the automorphism classes and characteristic subgroups of finite abelian $p$-groups. This topic was considered in 1905 and 1920 by G. A. Miller \cite{miller,miller2} and again, independently, in 1934 by Baer, who considered the more general case of periodic abelian groups \cite{baer}, and finally in 1935 by Birkhoff \cite{birkhoff_subabel}. The historical nature of these early works is such that they are not easy reading; the present author confesses that in many places the statements made in them, and particularly the proofs, do not always seem clear. A more recent treatment may be found in \cite{kerby_turner} (or \cite{kerby_masters}), which includes proofs of all the results below as well as a detailed description of the exceptional case $p=2$.

Throughout this section, $G$ will denote a finite abelian $p$-group: $$G=Z_{p^{\lambda_1}}\times Z_{p^{\lambda_2}}\times \cdots \times Z_{p^{\lambda_n}},$$ where $\lambda_1 \leq \lambda_2 \leq \cdots \leq \lambda_n$. We define $\h \lambda(G)$ to be the tuple $(\lambda_1,\dots,\lambda_n)$. For convenience, we will often use the convention $\lambda_0=0$. As we will be working extensively with such tuples of integers, it will be useful to introduce some notation for dealing with them:

\begin{defn}
Given tuples $\h a=(a_1,\dots,a_n)$ and $\h b=(b_1,\dots,b_n)$ with integer entries, define
\begin{align*}
\h a \leq \h b &\text{ if $a_i \leq b_i$ for all $i\in\{1,\dots,n\}$}\\
\h a \wedge \h b &= (\min\{a_1,b_1\}, \min\{a_2,b_2\}, \dots, \min\{a_n,b_n\})\\
\h a \vee \h b &= (\max\{a_1,b_1\}, \max\{a_2,b_2\}, \dots, \max\{a_n,b_n\})
\end{align*}
\end{defn}

Define $\Lambda(G)$ to be the set of tuples $$\Lambda(G)=\{\h a : \h 0 \leq \h a \leq \h \lambda(G)\}.$$ It is evident that $\Lambda(G)$, under the partial order $\leq$, forms a finite lattice in which $\wedge$ and $\vee$ are the greatest lower bound and least upper bound  respectively.

Given a tuple $\h a \in \Lambda(G)$, we define $T(\h a)$ to be the set of elements $g\in G$ for which the $i$th component of $g$ has order $p^{a_i}$:
$$T(\h a)=\{(g_1,g_2,\dots,g_n) \in G : |g_i|=p^{a_i}\text{ for all $i=1,\dots,n$}\}.$$
Note that the sets $T(\h a)$ partition the group $G$. If $g \in T(\h a)$, we say that the  \emph{type} of $g$ is $T(\h a)$.

The following is straightforward to prove:
\begin{lemma}\label{lemtype}
If two elements have the same type, then they are in the same automorphism class.
\end{lemma}

\begin{defn}
Given a type $T(\h a)$, the automorphism class of $G$ containing $T(\h a)$ is denoted $O(\h a)$.
\end{defn}

\begin{defn}\label{defcan}
A tuple $\h a=(a_1,\dots,a_n)$ and its corresponding type $T(\h a)$ are called \emph{canonical} if for all $i\in\{1,\dots,n-1\}$,
\begin{enumerate}
\item[(i)] $a_i \leq a_{i+1}$ and
\item[(ii)] $a_{i+1}-a_i \leq \lambda_{i+1}-\lambda_i$.
\end{enumerate}
\end{defn}

The definition of ``canonical" is justified by the following theorem.

\begin{thm}\label{unican}
Every automorphism class contains a unique canonical type. 
\end{thm}

\begin{example}
Let $G=Z_2\times Z_8=Z_2\times Z_{2^3}=\langle s\rangle\times\langle
t\rangle$. Then there are 6 automorphism classes of $G$, namely:
\begin{align*}
O(0,0)&=T(0,0)=\{1\},\\
O(0,1)&=T(0,1)=\{t^4\},\\
O(0,2)&=T(0,2)=\{t^2,t^6\},\\
O(1,1)&=T(1,1)\cup T(1,0)=\{s,st^4\},\\
O(1,2)&=T(1,2)=\{st^2,st^6\},\\
O(1,3)&=T(1,3)\cup T(0,3)=\{t,st,t^3,st^3,t^5,st^5,t^7,st^7\}.\\
\end{align*}
\end{example}

Having established a sufficiently detailed description of the automorphism classes for our purposes, we now turn to the characteristic subgroups. We let $\Char(G)$ denote the lattice of characteristic subgroups of $G$.

\begin{defn}
Given a tuple $\h a \in \Lambda(G)$, we define the subgroup $R(\h a)=\underset{\h b\leq \h a}\cup T(\h b)$ and call $R(\h a)$ the \emph{regular subgroup below} $\h a$.
\end{defn}
\begin{remark}
We use the term ``regular", following Baer (\cite{baer}). But this concept of regular should not be confused with the notion of a regular permutation group, nor of a regular $p$-group.
\end{remark}

\begin{thm}\label{canchar}
$R(\h a)$ is a characteristic subgroup if and only if $\h a$ is canonical.
\end{thm}

The following is easily verified by direct calculation:
\begin{thm}\label{Rlat}For any $\h a, \h b \in \Lambda(G)$,
\begin{enumerate}
\item[(i)] $R(\h a) \cap R(\h b) = R(\h a \wedge \h b)$
\item[(ii)] $\langle R(\h a), R(\h b) \rangle = R(\h a \vee \h b)$
\item[(iii)] $|R(\h a)|=p^{\sum_{i=1}^n a_i}$
\end{enumerate}
\end{thm}
From (i) and (ii), it follows that the regular characteristic subgroups form a sublattice of $\Char(G)$. Using Theorem \ref{canchar}, this then implies that if $\h a$ and $\h b$ are canonical tuples then so are $\h a \wedge \h b$ and $\h a \vee \h b$. (This is also not difficult to verify directly.) Thus the canonical tuples form a sublattice of $\Lambda(G)$; this sublattice will be denoted by $\m C(G)$.

The following theorem is of great importance to us:
\begin{thm}[Miller-Baer]\label{charreg}
If $p \neq 2$, then every characteristic subgroup of $G$ is regular.
\end{thm}

We then have the following important corollaries:
\begin{cor}\label{corcong}
If $p \neq 2$, then $\Char(G) \cong \m C(G)$.
\end{cor}

\begin{cor}[Miller-Baer-Birkhoff]\label{num_autos}
For any prime $p$, the number of automorphism classes of $G=Z_{p^{\lambda_1}}\times Z_{p^{\lambda_2}}\times \cdots \times Z_{p^{\lambda_n}}$ is $$\prod_{i=1}^n(\lambda_i-\lambda_{i-1}+1).$$
If $p\neq 2$, then this is also the number of characteristic subgroups of $G$.
\end{cor}
\begin{proof}
We apply Corollary \ref{corcong}, observing that the canonical tuples $\h a$ are precisely those which satisfy $a_{i-1} \leq a_i \leq a_{i-1}+\lambda_i-\lambda_{i-1}$ for each $i \in \{1,\dots,n\}$. Thus there are $\lambda_i-\lambda_{i-1}+1$ choices for each coordinate $a_i$, and the first statement follows. The second follows from the first and Theorems \ref{canchar} and \ref{charreg}.
\end{proof}

We define $\m W(G)$ to be the subalgebra of elements of $FG$ which are invariant under $\Aut(G)$. Thus, $\m W(G)$ is the maximum rational S-ring over $FG$; its basic sets are the automorphism classes of $G$. We then obtain one last corollary:

\begin{cor}\label{charind}
If $p \neq 2$, then $\q{\Char(G)}$ is a basis for $\m W(G)$.
\end{cor}
\begin{proof}
Clearly $\q{\Char(G)} \subseteq \m W(G)$. By Theorem \ref{unican}, the S-ring $\m W(G)$ is spanned by $\{\q{O(\h a)} : \h a \in \m C(G)\}$. For any canonical tuple $\h a\in\m C(G)$, we can write $$\q{O(\h a)} = \q{R(\h a)}-\sum_{\underset{\h b < \h a}{\h b \in \m C(G)}} \q{O(\h b)},$$
and, since $\q{R(\h a)} \in \q{\Char(G)}$, it follows by induction that each $\q{O(\h a)}$ is in the span of $\q{\Char(G)}$. Since $\dim(\m W(G))=|\Char(G)|$ by Corollary \ref{num_autos}, the result follows.
\end{proof}

\begin{example}\label{nonlatsring}
Let $G=Z_p\times Z_{p^3}$ for an odd prime $p$. Then the characteristic subgroups of $G$ are shown in Table \ref{char13}.
\end{example}

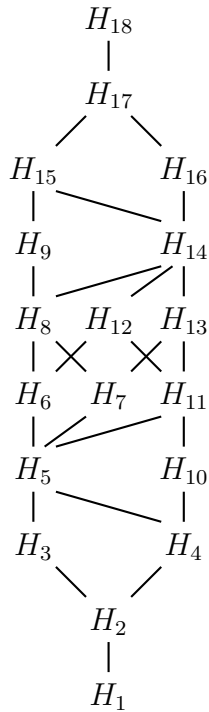
\begin{table}
\centering
\caption{Characteristic subgroups of $G=Z_p\times Z_{p^3}$ for odd prime $p$}\label{char13}

\begin{tabular}{lr}
\begin{tabular}{|l|l|} \hline
$H_1$&R(0,0)\\ \hline
$H_2$&R(0,1)\\ \hline
$H_3$&R(0,2)\\ \hline
$H_4$&R(1,1)\\ \hline
$H_5$&R(1,2)\\ \hline
$H_6$&R(1,3)\\ \hline
\end{tabular}
&
\begin{tabular}{r}
\psset{xunit=.2cm,yunit=.2cm,labelsep=2.5mm}
\begin{pspicture*}(-10,-10)(10,22)
\Rput[t](0,0){$H_1$}

\Rput[t](0,5){$H_2$}
\psline(0,-.5)(0,1.5)

\Rput[t](-5,10){$H_3$}
\Rput[t](5,10){$H_4$}
\psline(-1.5,4.5)(-3.5,6.5)
\psline(1.5,4.5)(3.5,6.5)

\Rput[t](0,15){$H_5$}
\psline(-3.5,9.5)(-1.5,11.5)
\psline(3.5,9.5)(1.5,11.5)

\Rput[t](0,20){$H_6$}
\psline(0,14.5)(0,16.5)
\end{pspicture*}
\end{tabular}
\end{tabular}
\end{table}

\begin{example}
Let $G=Z_p\times Z_{p^3}\times Z_{p^5}$ for an odd prime $p$. Then the characteristic subgroups of $G$ are shown in Table \ref{char135}. Note that the sublattice of $\Char(G)$ between $H_5=R(0,1,2)$ and $H_{14}=R(1,2,3)$ forms a cube; i.e., this sublattice is isomorphic to the boolean lattice $\m P(X)$ of subsets of a set $X$ of cardinality 3.
\end{example}

\begin{table}
\centering
\caption{Characteristic subgroups of $G=Z_p\times Z_{p^3}\times Z_{p^5}$ for odd prime $p$}\label{char135}

\begin{tabular}{lr}
\begin{tabular}{|l|l|} \hline
$H_1$&R(0,0,0)\\ \hline
$H_2$&R(0,0,1)\\ \hline
$H_3$&R(0,0,2)\\ \hline
$H_4$&R(0,1,1)\\ \hline
$H_5$&R(0,1,2)\\ \hline
$H_6$&R(0,1,3)\\ \hline
$H_7$&R(0,2,2)\\ \hline
$H_8$&R(0,2,3)\\ \hline
$H_9$&R(0,2,4)\\ \hline
$H_{10}$&R(1,1,1)\\ \hline
$H_{11}$&R(1,1,2)\\ \hline
$H_{12}$&R(1,1,3)\\ \hline
$H_{13}$&R(1,2,2)\\ \hline
$H_{14}$&R(1,2,3)\\ \hline
$H_{15}$&R(1,2,4)\\ \hline
$H_{16}$&R(1,3,3)\\ \hline
$H_{17}$&R(1,3,4)\\ \hline
$H_{18}$&R(1,3,5)\\ \hline
\end{tabular}
&
\begin{tabular}{r}
\psset{xunit=.2cm,yunit=.2cm,labelsep=2.5mm}
\begin{pspicture*}(-15,-5)(15,47)
\Rput[t](0,0){$H_1$}

\Rput[t](0,5){$H_2$}
\psline(0,-.5)(0,1.5)

\Rput[t](-5,10){$H_3$}
\Rput[t](5,10){$H_4$}
\psline(-1.5,4.5)(-3.5,6.5)
\psline(1.5,4.5)(3.5,6.5)

\Rput[t](-5,15){$H_5$}
\Rput[t](5,15){$H_{10}$}
\psline(-5,9.5)(-5,11.5)
\psline(5,9.5)(5,11.5)
\psline(3.5,9.5)(-3.5,11.5)

\Rput[t](-5,20){$H_6$}
\Rput[t](0,20){$H_7$}
\Rput[t](5,20){$H_{11}$}
\psline(-5,14.5)(-5,16.5)
\psline(-4.25,14.5)(-1.5,16.5)
\psline(-3.5,14.5)(3.5,16.5)
\psline(5,14.5)(5,16.5)

\Rput[t](-5,25){$H_8$}
\Rput[t](0,25){$H_{12}$}
\Rput[t](5,25){$H_{13}$}
\psline(-3.5,19.5)(-1.5,21.5)
\psline(-1.5,19.5)(-3.5,21.5)
\psline(3.5,19.5)(1.5,21.5)
\psline(1.5,19.5)(3.5,21.5)
\psline(-5,19.5)(-5,21.5)
\psline(5,19.5)(5,21.5)

\Rput[t](-5,30){$H_9$}
\Rput[t](5,30){$H_{14}$}
\psline(5,24.5)(5,26.5)
\psline(1.5,24.5)(4.25,26.5)
\psline(-3.5,24.5)(3.5,26.5)
\psline(-5,24.5)(-5,26.5)

\Rput[t](-5,35){$H_{15}$}
\Rput[t](5,35){$H_{16}$}
\psline(-5,29.5)(-5,31.5)
\psline(5,29.5)(5,31.5)
\psline(3.5,29.5)(-3.5,31.5)

\Rput[t](0,40){$H_{17}$}
\psline(-3.5,34.5)(-1.5,36.5)
\psline(3.5,34.5)(1.5,36.5)

\Rput[t](0,45){$H_{18}$}
\psline(0,39.5)(0,41.5)
\end{pspicture*}
\end{tabular}
\end{tabular}
\end{table}

The following theorem gives a generalization of preceding two examples which will be important to us in \S\ref{sec_main}:
\begin{thm}\label{latbool}
Let $X$ be a (finite) set containing $n$ elements, and let
$$G=Z_p \times Z_{p^3}\times\cdots\times\Z_{p^{2n-1}}.$$ 
Then there is an embedding $\psi$ of the boolean lattice $\m P(X)$ into $\Char(G)$. Further, $\psi$ has the property that for all subsets $Y_1, Y_2$ of $X$, $|Y_1|=|Y_2| \iff |\psi(Y_1)|=|\psi(Y_2)|$.
\end{thm}
\begin{proof}
We will map $\m P(X)$ onto the sublattice of $\Char(G)$ between $R(0,1,2,\dots,n-1)$ and $R(1,2,3,\dots,n)$ in the following way: Write $X=\{x_1,\dots,x_n\}$; then, for $Y \subseteq X$ define $\psi(Y)=R(\h a)$, where $\h a=(a_1,\dots,a_n)$ is given by
$$a_i=\begin{cases}i-1, & \text{if $x_i\notin Y$}\\ i, & \text{if $x_i\in Y$}\end{cases}$$
Note that $\psi$ is well-defined since each the tuple $\h a$, as thus defined, is always canonical, since the difference between two consecutive components is either 0, 1, or 2. We have
\begin{align*}
\psi(Y_1 \cap Y_2)&=R(\h a(Y_1 \cap Y_2))=R(\h a(Y_1) \wedge \h a(Y_2))\\
&= R(\h a(Y_1)) \cap R(\h a(Y_2)) = \psi(Y_1) \cap \psi(Y_2)
\end{align*} and
\begin{align*}
\psi(Y_1 \cup Y_2)&=R(\h a(Y_1 \cup Y_2))=R(\h a(Y_1) \vee \h a(Y_2)) \\
&= \langle R(\h a(Y_1)), R(\h a(Y_2)) \rangle = \langle \psi(Y_1), \psi(Y_2)\rangle,
\end{align*} so that $\psi$ is a lattice homomorphism. By construction $\psi$ is injective, so $\psi$ is an embedding. Theorem \ref{Rlat}(iii) shows that $|\psi(Y)|=p^{\frac{n(n-1)}2+|Y|}$, and the last claim follows.
\end{proof}

\section{Main Theorem}\label{sec_main}

We now turn to our main result. In this section, we take all S-rings over a coefficient field $F$ of characteristic $0$.

\begin{thm}\emph{(Main Theorem)}\label{mainthm}
Let $p$ be an odd prime. Every finite group can be represented as the automorphism group of a rational S-ring over an abelian $p$-group.
\end{thm}

The proof relies on three key ideas:

\begin{enumerate}
\item[(1)] Every finite group can be represented as the automorphism group of a finite distributive lattice,
\item[(2)] Every distributive lattice can be embedded in a boolean lattice, and
\item[(3)] The lattice of characteristic subgroups of the group $Z_p \times Z_{p^3}\times\cdots\times Z_{p^{2n-1}}$ contains a boolean sublattice of order $2^n$.
\end{enumerate}

The proof of the Main Theorem, in outline, goes as follows: Given a group $G$, by (1) there is a distributive lattice with automorphism group isomorphic to $G$. By (2), this lattice may be embedded in a boolean lattice, which by (3) may in turn be embedded as a sublattice of the lattice of characteristic subgroups of an appropriate abelian $p$-group $P$. Finally, using Theorem \ref{latsring}, we construct the rational S-ring over $P$ associated with this lattice and show that the automorphism group of this S-ring is isomorphic to G.

(1) was shown by Birkhoff in \cite{birkhoff}; a proof may also be found in 
\cite[\S7]{kerby_masters}, \cite{gratzer_birkhoff1}, or \cite{gratzer_birkhoff2} (see also \cite{foldes}). (2) is a standard result in lattice theory (also first proven by Birkhoff); a proof may be found in \cite[Theorem 5.12]{davey}, \cite[20.1]{hermes}, \cite[11.3]{crawley}, or \cite[Theorem 6.6]{kerby_masters}. We have already shown (3) in Theorem \ref{latbool} above. So all that remains is to verify the last step of our argument. Before doing this, however, we will need to give a more explicit description of the embedding of (2):

Recall that an element $j$ of a lattice $\m L$ is \emph{join-reducible} if there exist $x<j$ and $y<j$ with $j=x\vee y$. Otherwise $j$ is said to be \emph{join-irreducible}. 

\begin{thm}[Birkhoff]\label{latdist}Let $\m L$ be a finite distributive lattice and let $J \subseteq \m L$ be the set of join-irreducible elements. Then the map $\phi(x)=\{j \in J : j \leq x\}$ is a embedding of $\m L$ into $\m P(J)$.
\end{thm}

In what follows, we will need the following fact:

\begin{thm}\label{latsize}
In Theorem \ref{latdist}, $|\phi(\alpha(x))|=|\phi(x)|$ for all $\alpha \in \Aut(\m L)$ and $x\in \m L$.
\end{thm}
\begin{proof}
Since lattice isomorphisms permute the join-irreducible elements among themselves, we have
\begin{align*}
\phi(\alpha(x)) &= \{j \in J : j \leq \alpha(x)\}
= \{\alpha(j) \in J : \alpha(j) \leq \alpha(x)\}\\
&= \{\alpha(j) \in J : j \leq x \} = \alpha(\phi(x)),
\end{align*}
hence $|\phi(\alpha(x))|=|\alpha(\phi(x))|=|\phi(x)|$. 
\end{proof}

Now we are able to prove the Main Theorem. Given a finite group $G$,
by (1) there is a
finite distributive lattice $\m D$ with $\Aut(D) \cong G$. By
Theorem \ref{latdist}, there is an embedding $\phi$ of $\m D$ into
the boolean lattice $\m P(J)$, where $J$ is the set of
join-irreducible elements of $\m D$. In turn, by Theorem
\ref{latbool}, there is an embedding $\psi$ of $\m P(J)$ into the
lattice of characteristic subgroups $\Char(P)$ of an appropriate
abelian $p$-group $P$. Let $\m L=\psi(\phi(D))\cup\{1\}\cup\{P\}$. Then by Theorem \ref{latsring}, $F\q{\m L}$ is an S-ring over $P$. To prove the theorem, we will show that $\Aut(F\q{\m L}) \cong G$.

Since every automorphism of $\m L$ fixes $1$ and $P$ and hence restricts to an automorphism of $\psi(\phi(D))$, it follows that $\Aut(\psi(\phi(D))) \cong \Aut(\m L)$. So since $D \cong \psi(\phi(D))$, we then have $G \cong \Aut(D) \cong \Aut(\psi(\phi(D))) \cong \Aut(\m L)$. So 
it suffices to show $\Aut(\m L) \cong \Aut(F\q{\m L})$.

We define a map $\chi : \Aut(\m L) \to \Aut(F\q{\m L})$ by $\chi(\beta) : \q H \mapsto \q{\beta(H)}$ for $H\in\m L$. We need to justify that $\chi$ is well-defined, i.e., that $\chi(\beta)$ actually is an S-ring automorphism of $\m L$. After that, it will be clear that $\chi$ is an injective homomorphism. Proving the surjectivity of $\chi$ will then complete the proof.

Let $\beta\in \Aut(\m L)$. Since Corollary \ref{charind} implies that $\q{\m L}$ is linearly independent and hence forms a basis for $F\q{\m L}$, it is clear that $f=\chi(\beta)$ is a well-defined bijective linear map from $F\q{\m L}$ to itself. To show that $f$ is an S-ring automorphism, we just need to show that it preserves the Hadamard and ordinary products. Applying Theorem \ref{latsring}, we have
\begin{align*}
f(\q H \circ \q K) &= f(\q{H\cap K}) = \q{\beta(H\cap K)} = \q{\beta(H)\cap\beta(K)} \\
&= \q{\beta(H)} \circ \q{\beta(K)} = f(\q H) \circ f(\q K),
\end{align*}
so $f$ preserves the Hadamard product. Before we can show $f$ preserves the ordinary product, we first need the result that for all $H \in \m L$, $|\beta(H)|=|H|$. For $H=1$ or $H=P$ this is trivially true since in these cases $\beta(H)=H$. For any other $H$, we may write $H=\psi(\phi(x))$ for some $x\in D$. Now, letting $\alpha\in\Aut(D)$ be the composite function $\alpha=\phi^{-1}\psi^{-1}\beta\psi\phi$, we see that showing $|\beta(H)|=|H|$ is equivalent to showing $|\psi(\phi(\alpha(x)))|=|\psi(\phi(x))|$, which by Theorem \ref{latbool} is equivalent to showing $|\phi(\alpha(x))|=|\phi(x)|$, which in turn holds by Theorem \ref{latsize}. So we have proven $|\beta(H)|=|H|$. In particular, if $H,K\in\m L$, we have
$$|H\cap K| = |\beta(H\cap K)|=|\beta(H)\cap\beta(K)|,$$
and from this, by applying Theorem \ref{latsring}, we obtain 
\begin{align*}
f(\q H\ \q K) &= f(|H \cap K|\q{HK})
= |H\cap K| f(\q{HK})\\
&= |H\cap K| \q{\beta(HK)} 
= |H\cap K| \q{\beta(H)\beta(K)}\\
&= |\beta(H)\cap\beta(K)|\q{\beta(H)\beta(K)}
= \q{\beta(H)}\ \q{\beta(K)}
= f(\q H)f(\q K),
\end{align*}
so $f$ preserves the ordinary product, which proves $f\in\Aut(F\q{\m L})$.

Now we only need to show that $\chi$ is surjective. 
Suppose $\gamma$ is any S-ring automorphism of $F\q{\m L}$. For any $H \in \m L$, we have $\gamma(\q H)=\q K$ for some
subgroup $K\leq G$ by Theorem \ref{isosimp}(ii). Since $F\q{\m L}$ is a rational S-ring, $K$ is
necessarily characteristic. By the linear independence of 
characteristic subgroups (Corollary \ref{charind}),
we must have $K \in \m L$. Thus $\gamma$ permutes the basis elements
$\{\q H : H \in \m L\}$ of $F\q{\m L}$. We can then define a bijection
$\beta : \m L \to \m L$ by setting $\beta(H)=K$ where $K$ is the
subgroup of $G$ such that $\gamma(\q H)=\q K$, so that
$\q{\beta(H)}=\gamma(\q H)$ for all $H\in \m L$. Once we show that $\beta$ is a lattice automorphism of $\m L$, we will have $\gamma=\chi(\beta)$, which will complete the proof.

Theorem
\ref{latsring} implies
\begin{align*}
\q{\beta(H \cap K)}
&= \gamma(\q{H \cap K})
=\gamma(\q H \circ \q K)\\
&=\gamma(\q H) \circ \gamma(\q K)
=\q{\beta(H)} \circ \q{\beta(K)}
=\q{\beta(H) \cap \beta(K)}
\end{align*}
so that $\beta(H\cap K)=\beta(H)\cap\beta(K)$. Observe that $|\beta(H)|=|H|$ by Theorem \ref{isosimp}(i), hence 
$$|H\cap K|=|\beta(H\cap K)|=|\beta(H)\cap\beta(K)|,$$
so a further application of Theorem \ref{latsring} implies
\begin{align*}
\q{\beta(HK)}
&= \gamma(\q{HK})
= \gamma\left(\frac1{|H\cap K|}\q H\ \q K\right)
= \frac1{|H\cap K|}\gamma(\q H)\gamma(\q K)\\
&= \frac1{|H\cap K|}\q{\beta(H)}\ \q{\beta(K)}
= \frac1{|\beta(H)\cap\beta(K)|}\q{\beta(H)}\ \q{\beta(K)}
= \q{\beta(H)\beta(K)}
\end{align*}
so that $\beta(HK)=\beta(H)\beta(K)$. This proves that $\beta$ is a lattice automorphism of $\m L$, as desired.

\section{Automorphisms of S-rings over cyclic groups}\label{sec_cyc}

In this section, we again take all S-rings over a coefficient field $F$ of characteristic zero.

In \cite{leung_man2}, Leung and Man give a recursive classification
of all S-rings over cyclic groups. We give a brief description of
this classification. They give three basic methods of
constructing S-rings over a group $G$:

\begin{enumerate}
\item[(I)] Given a subgroup $\Omega \leq \Aut(G)$, let $T_1,\dots,T_n$
be the orbits of $\Omega$ acting on $G$. Then $T_1,\dots,T_n$ form a
Schur partition of $G$.

\item[(II)] Suppose $G=H\times K$ for nontrivial subgroups $H, K \leq G$, and suppose $S_H$ is an S-ring over $H$ with basic sets
$C_1,\dots,C_h$ and $S_K$ is an S-ring over $K$ with basic sets
$D_1,\dots,D_k$. Then the product sets $C_iD_j$, $1 \leq i \leq h, 1
\leq j\leq k$, form a Schur partition of $G$.

\item[(III)] Suppose $H$ and $K$ are nontrivial, proper subgroups of $G$ with $H \leq
K$ and $H \unlhd G$, and let $S_K$ be an S-ring over $K$ with basic
sets $C_1,\dots,C_k$ and $S_{G/H}$ be an S-ring over $G/H$ with
basic sets $D_1,\dots,D_k$, and suppose that $\pi(S_K)=F(K/H)\cap
S_{G/H}$, where $\pi: G \to G/H$ is the natural projection map,
extended to a natural projection map of the group algebra $FG$ onto
$F(G/H)$. Then
$$G=C_1\cup\cdots\cup C_k\cup\{ \pi^{-1}(D_i) : i\in\{1,\dots,k\}, D_i \nsubseteq K/H
\}$$ forms a Schur partition of $G$.
\end{enumerate}

In the case of cyclic groups, the S-ring constructed in (I) is called a \emph{cyclotomic} S-ring (a more general definition of cyclotomic S-rings, together with some of their applications, is found in \cite{muzychuk_ponomarenko}).
The S-ring constructed in (II) is denoted $S_H \cdot S_K$ and is
called the \emph{dot product} (also denoted $S_H \otimes S_K$ and called the \emph{tensor product}) of $S_H$ and $S_K$. The S-ring
constructed in (III) is denoted $S_K \wedge S_{G/H}$ and is called
the \emph{wedge product} (or \emph{generalized wreath product}) of $S_K$ and $S_{G/H}$. 

The main theorem of Leung and Man (\cite[Theorem 3.7]{leung_man2}) may then be stated as follows:

\begin{thm}
Every nontrivial S-ring over a cyclic group $G$ is either cyclotomic, a dot product, or a wedge product.
\end{thm}

Note that the constructions (I), (II), and (III) can be used to
produce S-rings over an arbitrary group $G$, but if $G$ is not
cyclic then it is not necessarily true that all S-rings can be
constructed using these methods. For instance, over elementary abelian $p$-groups we have the following (for proofs, see \cite[Theorem 8.4, Example 8.5]{kerby_masters}):

\begin{thm}Let $p$ be a prime with $p\geq 5$. Then there are S-rings
over $Z_p\times Z_p$ which cannot be constructed as in (I), (II), or (III).
\end{thm}

\begin{thm}There are S-rings over $Z_3\times Z_3 \times Z_3 \times Z_3$ which cannot be constructed as in (I), (II), or (III).
\end{thm}

In contrast, all S-rings over $Z_3 \times Z_3$ and $Z_3 \times Z_3 \times Z_3$ can be constructed as in (I). In \cite[Question 8.6]{kerby_masters}, the question was asked of whether all S-rings over an elementary abelian 2-group can be constructed as in (I). We can now give a negative answer to this:

\begin{thm}There are S-rings over $Z_2^6$ which cannot be constructed as in (I), (II), or (III).
\end{thm}
\begin{proof}
Consider $G=Z_2^6$ as the additive group of the finite field $F_{64}$. Since $x^6+x^4+x^3+x+1$ is primitive over $\mathbb{F}_2$, $F_{64}^\times$ has a generator $\omega$ with $\omega^6+\omega^4+\omega^3+\omega+1=0$. The action of multiplication by $\langle \omega^9 \rangle$ partitions $F_{64}^\times$ into 9 orbits (each of size 7),
$$C_i=\{\omega^{i+9j} : j \in \{0..6\}\},\quad i=0,\dots,8.$$
We have found using MAGMA \cite{magma} that
$$\{0\},C_0\cup C_1\cup C_2\cup C_3\cup C_4, C_5\cup C_6\cup C_7\cup C_8$$
forms a Schur partition of $G$ whose corresponding S-ring $S$ cannot be constructed as in (I), (II), or (III) (For (II) and (III), this is fairly obvious: since basic sets have size 35, 28, and 1, it is clear that $S$ has no nontrivial proper $S$-subgroup, i.e., $S$ is a \emph{primitive} S-ring.)
\end{proof}
\begin{remark}
An exhaustive enumeration using MAGMA shows that all S-rings over $Z_2^n$ for $n \leq 4$ can be constructed as in (I). The question in the case $n=5$ remains open; this remaining case would probably still be feasible to answer by exhaustive enumeration, if one reduced the computation by taking advantage of the symmetry of the group.
\end{remark}

We now turn to the main result of this section:

\begin{thm}\label{autcyc}
If $S$ is an S-ring over a cyclic group $Z_n$, then $\Aut(S)$ is abelian.
\end{thm}

Before proving this, we need a few elementary lemmas, proofs of which may be found in \cite[Lemmas 8.8, 8.9, 8.10]{kerby_masters}. A version of
Lemma \ref{lemwie} can be found in \cite[Theorem 23.9]{wielandt},
while more general versions of Lemmas \ref{autexp} and
\ref{localstr} can be found in \cite[Propositions 3.1, 3.4]{muzychuk2}.
\begin{lemma}[Wielandt]\label{lemwie}
Suppose $S$ is an S-ring over an abelian group $G$ and $m$ is an
integer relatively prime to $|G|$. If $\q C \in S$, then also $\q
C^{(m)}\in S$. Moreover, if $\q C$ is a basic element of $S$ then
$\q C^{(m)}$ is also a basic element of $S$.
\end{lemma}

\begin{lemma}\label{autexp}
Suppose $S$ is an S-ring over an abelian group $G$ and
$\phi\in\Aut(S)$. Then for any simple quantity $\q C\in S$ and any
integer $m$ relatively prime to $|G|$, we have
$\phi(C^{(m)})=\phi(C)^{(m)}$.
\end{lemma}

\begin{lemma}\label{localstr}
Suppose $S$ is a cyclotomic S-ring over $Z_n$. Let $\phi$ be an
automorphism of $S$. Then for any basic set $T$ of $S$, there is an
integer $m$ relatively prime to $n$ such that $\phi(T)=T^{(m)}$.
\end{lemma}
\begin{remark}
Lemma \ref{lemwie} says that over an abelian group, every Cayley isomorphism is actually an automorphism of the S-ring. Lemma \ref{autexp} says that the Cayley automorphisms are in the center of the automorphism group of the S-ring. Finally, Lemma \ref{localstr} says that for a cyclotomic S-ring over a cyclic group, every S-ring automorphism
``locally" behaves like a Cayley automorphism. (In Lemma \ref{localstr}, the assumption that the S-ring be cyclotomic is actually unnecessary. However, the generalization omitting this hypothesis is a much deeper result which we will not need to use directly; it may be found in \cite[Theorem 1.1']{muzychuk2}, where it plays a key role in the proof of Muzychuk's result cited in Theorem \ref{cyciso}.)
\end{remark}

\begin{proof}[Proof of Theorem \ref{autcyc}]
If $S$ is a trivial S-ring, then $\Aut(S)$ is also trivial, hence
abelian, and we are done. So first consider the case that $S$ is
cyclotomic. Each basic set of $S$ then consists of elements of the
same order. Let $\m T_d$ be the collection of basic sets of $S$
containing elements of order $d$. We can consider $\Aut(S)$ as a
permutation group acting on the basic sets of $S$, and by Lemma
\ref{localstr}, for each $d$, $\Aut(S)$ permutes the basic sets of
$\m T_d$ among themselves. If we let $A_d$ denote the restriction of
$\Aut(S)$ to $\m T_d$, then $\Aut(S)$ is a subdirect product of all
the $A_d$'s, so it suffices to show that each $A_d$ is abelian. Fix
a divisor $d$ of $n$, and let $T$ be a basic set in $\m T_d$. For
any $k$ relatively prime to $n$, $T^{(k)}$ is another basic set of
$\m T_d$ by Lemma \ref{lemwie}, and every basic
set of $\m T_d$ has this form, for if $T'$ is any basic set in $\m
T_d$, there is an integer $l$ relatively prime to $n$ such that
$T^{(l)}\cap T' \neq \emptyset$, hence $T^{(l)}=T'$. Now, by Lemma
\ref{localstr}, $\phi(T)=T^{(m)}$ for some $m$ relatively prime to
$n$. It follows by Lemma \ref{autexp} that for any basic set
$T^{(k)} \in \m T_d$,
$$\phi(T^{(k)})=\phi(T)^{(k)}=(T^{(m)})^{(k)}=(T^{(k)})^{(m)}.$$
Thus $\phi(T')=(T')^{(m)}$ for any basic set $T' \in \m T_d$. From
this it is evident that $A_d$ is abelian.

Now suppose $S$ is a dot product $S_H\cdot S_K$. By induction we may
assume $\Aut(S_H)$ and $\Aut(S_K)$ are abelian. Let $\phi$ be an
element of $\Aut(S)$. Since $H$ is the unique subgroup of order $|H|$ in $G$,  Theorem \ref{isosimp}(i,ii) implies $\phi(\q H)=\q H$, and similarly $\phi(\q K)=\q K$. So $\phi(S_H)$ is an S-ring over $H$ which
is isomorphic to $S_H$. By the result of Muzychuk cited in Theorem
\ref{cyciso}, the only such S-ring is $S_H$ itself, so we must have
$\phi(S_H)=S_H$ and likewise $\phi(S_K)=S_K$. So $\phi|_{S_H} \in
\Aut(S_H)$ and $\phi|_{ S_K}\in\Aut(S_K)$. Given another
automorphism $\psi \in\Aut(S)$, and any basic set $CD$ of $S$, where
$C$ is a basic set of $S_H$ and $D$ is a basic set of $S_K$, we have
$\phi(\psi(\q C))=\psi(\phi(\q C))$ and $\phi(\psi(\q
D))=\psi(\phi(\q D))$ since $\Aut(S_H)$ and $\Aut(S_K)$ are abelian,
hence \begin{align*} \phi(\psi(\q{CD}))&=\phi(\psi(\q C\ \q
D))=\phi(\psi(\q C)\psi(\q D))=\phi(\psi(\q C))\phi(\psi(\q D))\\
&=\psi(\phi(\q C))\psi(\phi(\q D))=\psi(\phi(\q C)\phi(\q
D))=\psi(\phi(\q C\ \q D))=\psi(\phi(\q{CD}))
\end{align*}
which proves that $\phi$ and $\psi$ commute, so $\Aut(S)$ is
abelian.

Finally suppose $S$ is a wedge product $S_K \wedge S_{G/H}$. As above,
given an automorphism $\phi \in \Aut(S)$, we have $\phi|_{S_K}
\in\Aut(S_K)$. Now define $F$-linear maps $\pi : FG \to F(G/H)$
and $\pi': F(G/H) \to FG$ by
\begin{align*}
&\pi: g \mapsto gH \\
&\pi': gH \mapsto g\q H.
\end{align*}
The following relations are easily checked:
\begin{align*}\label{pipi}
\pi(\pi'(x)) &= |H|x \\
\pi(xy) &= \pi(x)\pi(y)\\
\pi'(xy) &= \frac1{|H|}\pi'(x)\pi'(y)
\end{align*}
Consider the linear map $\phi^* : S_{G/H} \to S_{G/H}$ given by the composite function $\phi^* =\frac1{|H|}\pi\phi\pi'$. We need to justify that this is well-defined. 
Note that if $B$ is a basic set of $S_{G/H}$ then $\pi'(\q B)=\q C$ where $C$ is a basic set of $S_G$ and $C$ is a union of cosets of $H$, so that $\q C\ \q H=|H|\q C$. Applying $\phi$ to boths sides gives $\phi(\q C)\q H=|H|\phi(\q C)$, so that the basic set $D=\phi(C)$ of $S_G$ is also union of cosets of $H$. Then $\pi(\q D)=|H|\q E$ for a basic set $E$ of $S_{G/H}$. Putting this together, we have
$$\phi^*(\q B)=\frac1{|H|}\pi(\phi(\pi'(\q B)))=\frac1{|H|}\pi(\phi(\q C))=\frac1{|H|}\pi(\q D)=\q E,$$
so that $\phi^*$ is well-defined and maps basic quantities to basic quantities. Now, we have
\begin{align*}
\left(\frac1{|H|}\pi'\phi^{-1}\pi\right)\phi^* &= \frac1{|H|^2}\pi'\phi^{-1}\pi\pi'\phi\pi 
= \frac1{|H|}\pi'\phi^{-1}\phi\pi=\frac1{|H|}\pi'\pi=1_{FG},
\end{align*}
so $\phi^*$ is a bijection. Finally,
\begin{align*}
\phi^*(xy)&=\frac1{|H|}(\pi\phi\pi')(xy)=\frac1{|H|^2}\pi(\phi(\pi'(x)\pi'(y)))\\
&=\frac1{|H|^2}\pi(\phi(\pi'(x)))\pi(\phi(\pi'(y)))=\phi^*(x)\phi^*(y)
\end{align*}
so $\phi^*$ is an S-ring automorphism of $S_{G/H}$.

By induction we may assume
$\Aut(S_K)$ and $\Aut(S_{G/H})$ are abelian. Then, given another
automorphism $\psi\in\Aut(S)$ and a basic set $C$ of $S$, to complete the proof, it suffices to show $\phi(\psi(\q C))=\psi(\phi(\q C))$. We have
two cases: If $C\subseteq K$, then
$$\phi(\psi(\q C))=\phi|_{S_K}(\psi|_{S_K}(\q C))=\psi|_{S_K}(\phi|_{S_K}(\q C))=\psi(\phi(\q C))$$
and we are done. Suppose instead that $C\nsubseteq K$. Then $C$ is a
union of cosets of $H$, hence $\pi'(\pi(\q C))=|H|\q C$. As above,
$\phi(C)$ is also a union of cosets of $H$. Similarly, so are $\psi(C)$, $\phi(\psi(C))$, and $\psi(\phi(C))$. It follows that
\begin{align*}
(\pi'\pi\phi)(\q C) &= \pi'(\pi(\q{\phi(C)})) =
|H|\q{\phi(C)}=|H|\phi(\q C)
\end{align*}
and likewise
\begin{align*}
(\pi'\pi\psi)(\q C) &= |H|\psi(\q C) \\
(\pi'\pi\phi\psi)(\q C) &= |H|(\phi\psi)(\q C) \\
(\pi'\pi\psi\phi)(\q C) &= |H|(\psi\phi)(\q C)
\end{align*}
From all of this it follows that
\begin{align*}(\phi\psi)(\q C)&=\frac1{|H|}(\pi'  \pi\phi\psi)(\q C)
=\frac1{|H|^2}(\pi'  \pi\phi\pi'\pi\psi)(\q
C)\\
&=\frac1{|H|^3}(\pi'
\pi\phi\pi'\pi\psi\pi'\pi)(\q C)\\
&=\frac1{|H|}(\pi'  \phi^*\psi^*\pi)(\q C) =
\frac1{|H|}(\pi'
\psi^*\phi^*\pi)(\q C)\\
&=\frac1{|H|^3}(\pi'
\pi\psi\pi'\pi\phi\pi'\pi)(\q C)\\
&=\frac1{|H|^2}(\pi'
\pi\psi\pi'\pi\phi)(\q C)=\frac1{|H|}(\pi'
\pi\psi\phi)(\q C)=(\psi\phi)(\q C),
\end{align*}
so that $\phi$ and $\psi$ commute, as desired.
\end{proof}

We observe that Theorem \ref{autcyc} is false if the coefficient field $F$, or more generally the coefficient ring $R$, has nonzero characteristic:

\begin{example}\label{ex_nzchar}
Let $R$ have characteristic $n>0$. Set $G=Z_{4n}=\langle t\rangle$ and
$H=Z_n\leq G$. Define $S_H$ to be the S-ring $S=S_H \wedge S_{G/H}$
where $S_H$ is the trivial S-ring over $H$ and $S_{G/H}$ is the full
group algebra $R(G/H)$. Then $S$ has five basic sets
$$\{1\},Z_n-\{1\},T_1,T_2,T_3$$ where $T_i=t^iZ_n$. Then in $RG$ we have $\q T_i\q T_j=nt^{i+j}\q Z_n=0$ for all
$i,j\in\{1,2,3\}$ while $(\q Z_n-1)\q T_i = -\q T_i$. Thus $\Aut(S)
\cong \Sym_3$ is non-abelian. Over a ring with characteristic zero,
this same Schur partition gives an S-ring with automorphism group
isomorphic to $Z_2$.
\end{example}

We observe that the Leung-Man classification of S-rings over cyclic groups does not hold if the coefficient field has nonzero characteristic (see, e.g., \cite[Example 1.11]{kerby_masters}), but Example \ref{ex_nzchar} shows that this is not the only reason that Theorem \ref{autcyc} fails in this case. This suggests two problems, both of which we can expect to be difficult:

\begin{problem}
Classify the S-rings over cyclic groups for coefficient rings of nonzero characteristic.
\end{problem}

\begin{problem}
Describe the automorphism groups of such S-rings.
\end{problem}

\section{Converse to Muzychuk's Theorem}\label{sec_conv}
A remarkable theorem of Muzychuk states:

\begin{thm}[\cite{muzychuk2}]\label{cyciso}Two S-rings over a cyclic group $Z_n$ are isomorphic if and only if they are identical. \end{thm}

We prove a converse to this result. But first we need the following lemma:

\begin{lemma}\label{lemnonchar}
Let $G$ be a finite group which is not cyclic. Then $G$ has a subgroup which is not characteristic.
\end{lemma}
\begin{proof}
By way of contradiction, suppose every subgroup of $G$ is
characteristic. Then in particular every subgroup of $G$ is normal.
If $G$ is non-abelian, then $G$ is a Hamiltonian group (i.e., a nonabelian group in which every subgroup is normal) and we may write $G = Q \times A$ where $Q$ is an
8-element quaternion group $\langle i,j\rangle$ and $A$ is abelian
\cite[9.7.4]{scott}. But in this case $\langle i \rangle$ is a
subgroup of $G$ which is not characteristic, since there is an
automorphism of $Q$ mapping $\langle i \rangle$ to $\langle j
\rangle$ and this automorphism extends to an automorphism of $G$.
Therefore $G$ must be abelian.

Since $G$ is not cyclic, some Sylow $p$-subgroup of $G$ is not
cyclic and, by the Fundamental Theorem of finitely-generated abelian
groups, we may write $G = \langle t \rangle \times \langle s \rangle
\times A$ where $|t|=p^a$ and $|s|=p^b$ for some $a$ and $b$ where
$1 \leq a \leq b$. Then $\langle s \rangle$ is not characteristic,
since an automorphism $\phi$ is determined by setting $\phi(s)=ts$,
$\phi(t)=t$, and $\phi(a)=a$ for all $a \in A$.
\end{proof}

We remark that, by a similar method of proof, Lemma \ref{lemnonchar}
may be extended to infinite non-abelian groups and to finitely
generated abelian groups. However, there are non-cyclic infinitely
generated abelian groups in which every subgroup is characteristic,
an example being the direct sum $\underset{\text{$p$ prime}}\sum
Z_p$.

\begin{thm}\label{thm_conv}
Let $G$ be a finite group which is not cyclic. Then there exist distinct Cayley-isomorphic S-rings $S_1$ and $S_2$ over $G$.
\end{thm}
\begin{proof}
By Lemma \ref{lemnonchar}, let $H$ be a subgroup of $G$ which is not
characteristic. Choose some $\phi \in \Aut(G)$ such that $\phi(H)
\neq H$. Then $S_1=F\{1,\q H,\q G\}$ and $S_2=F\{1,\q{\phi(H)},\q G\}$
are S-rings over $G$ which are Cayley-isomorphic. We only need to
verify that they are distinct. The basic quantities of $S_1$ are
$\{1,\q H-1,\q G-\q H\}$ while the basic quantities of $S_2$ are
$\{1,\q{\phi(H)}-1,\q G-\q{\phi(H)}\}$. If $S_1=S_2$ then the basic
quantities of the two S-rings must be the same (in some order), so
either $\q H-1=\q{\phi(H)}-1$ or $\q H-1=\q G-\q{\phi(H)}$. The former
is impossible since $H \neq \phi(H)$. The latter would imply $G=H
\cup \phi(H)$, which is impossible, since no group is the union of
two proper subgroups.
\end{proof}

\section{Some examples}\label{sec_ex}

\begin{example}\label{nonlatsring}
Let $G=Z_p\times Z_{p^3}$ for an odd prime $p$. Let $H_1,\dots,H_6$ be the characteristic subgroups of $G$, in the order shown in Table \ref{char13}. Using Theorem \ref{folklore}, it is easy to check that $S=F\{1,\q H_2,\q H_3+\q H_4,\q H_5,\q G\}$ is a rational S-ring. We show that $S$ cannot be constructed as in Theorem \ref{latsring}. Suppose $S=F\q{\m L}$ for some lattice $\m L$. By Corollary \ref{charind}, the elements $\{\q H : H \in \m L\}$ are linearly independent, hence form a basis for $S$. So $\dim S=|\m L|$. Now $\dim S=5$, yet, by applying Corollary \ref{charind} again, it is easy to see that $1, H_2, H_5$, and $G$ are the only four subgroups of $G$ which are $S$-sets, hence $|\m L|\leq 4$, a contradiction.
\end{example}

Over other abelian $p$-groups, it is easy to construct many rational S-rings similar to Example \ref{nonlatsring}, where the basic quantities are sums of characteristic subgroups, chosen in such a way as to ensure closure under the Hadamard and ordinary product; but, thus far, it has not been possible to extend this construction to give a complete classification of rational S-rings over abelian $p$-groups. Some of the difficulty is indicated by the following example:

\begin{example}
Let $G=Z_p\times Z_{p^3}\times Z_{p^5}$ where $p$ is any prime. Let $H_1,\dots,H_{18}$ be the regular characteristic subgroups of $G$, as in Table~\ref{char135}. By direct computation, one can show using Corollary~\ref{folklore_cor} that $$S=F\{1,\q H_5,\q H_6+\q H_7-\q H_{8}-\q H_{11},\q H_8+3\q H_{11}-\q H_{12}-\q H_{13},\q H_{14},\q G\}$$ is an S-ring over $G$ if and only if $p=3$. It is of course also possible to present $S$ in terms of its basic elements:
$$S=F\{1,\q O_2+\q O_3+\q O_4+\q O_5,\q O_6+\q O_7+\q O_{14},
\q O_{12}+\q O_{13},\q O_8+\q O_{10}+\q O_{11},\q O_9+\q O_{15}+\q O_{16}+\q O_{17}+\q O_{18}\},$$
where $O_i=O(\h a)$ for $H_i=R(\h a)$. This S-ring does not have a basis consisting of sums of characteristic subgroups. This example also shows that choice of prime $p$ can make a difference in determining whether a partition of automorphism classes of $G$ is a Schur partition or not, and that it is not merely a question of whether $p=2$. For a further discussion of this and related examples, see \cite[Example 5.28]{kerby_masters} (But note that the initial basis given there for $S$ is erroneous; we have given a correct basis above).
\end{example}

Such examples lead us to a conjecture:

\begin{conj}
Let $(\lambda_1,\dots,\lambda_n)$ be a tuple of integers, $1 \leq \lambda_1 \leq \cdots \leq \lambda_n$. Then, as $p \to \infty$, the set of Schur partitions of the abelian $p$-group
$$G=Z_{p^{\lambda_1}} \times \cdots \times Z_{p^{\lambda_n}}$$
is eventually constant, and for sufficiently large $p$, every rational S-ring over $G$ has a basis consisting of sums of characteristic subgroups (where, here we are abusing notation slightly by thinking of the Schur partitions as partitions of $\m C(G)$, rather than of $G$ itself).

\end{conj}

One natural approach to classifying rational S-rings over abelian $p$-groups would involve (1) answering this conjecture, (2) giving an explicit description of which sums of characteristic subroups are allowed, and (3) attempting to understand the apparently ``exceptional" rational S-rings which occur when $p$ is small relative to $n$.

\bibliographystyle{plain}
\bibliography{automorphisms_of_schur_rings}

\end{document}